\documentclass[a4,12pt]{amsart}%
\usepackage{amsmath,amssymb,amsfonts,enumerate,amsthm, amscd,}
\usepackage[all]{xy}
\usepackage{graphicx}
\usepackage{amsmath}
\usepackage{amsfonts}
\usepackage{amssymb}
\usepackage{hyperref}%
\usepackage{pst-node}
\usepackage{tikz-cd} 
\providecommand{\U}[1]{\protect\rule{.1in}{.1in}}
\newtheorem{thm}{Theorem}[section]
\newtheorem{cor}[thm]{Corollary}
\newtheorem{lem}[thm]{Lemma}
\newtheorem{prop}[thm]{Proposition}
\newtheorem{defn}[thm]{Definition}

\newtheorem{rem}[thm]{Remark}

\theoremstyle{definition}

\theoremstyle{remark}

\theoremstyle{Definition and Notation}

\begin{document}

\title{ On $\jmath$-Artinian Modules over Commutative Rings}

	\author{Dilara Erdemir}
    \address{Dilara Erdemir, Department of Mathematics, Yildiz Technical University, Istanbul, Turkey.
$$E-mail\ address:\ dilaraer@yildiz.edu.tr$$}

\author{Najib Mahdou}
	\address{Najib Mahdou\\Department of Mathematics, Faculty of Science and Technology of Fez, Box 2202,
		University S.M. Ben Abdellah Fez, Morocco.
		$$E-mail\ address:\ mahdou@hotmail.com$$}
	
	\author{El Houssaine Oubouhou}
	\address{El Houssaine Oubouhou, Department of Mathematics, Faculty of Science and Technology of Fez, Box 2202,
		University S.M. Ben Abdellah Fez, Morocco.
		$$E-mail\ address:\ hossineoubouhou@gmail.com$$}

	\author{\"{U}NSAL TEK\.{I}R}
\address{\"{U}nsal Tek\.{i}r, Department of Mathematics, Marmara University, Istanbul, Turkey. 
$$E-mail\ address:\ utekir@marmara.edu.tr$$}

\maketitle
\bigskip
	\noindent
	{\small{\bf ABSTRACT.}
		%Researchers defined $\jmath$-Artinian rings in their paper \cite{jArt} and obtained important results on this new class of rings. Inspired by the definition and findings in this work, we introduce $\jmath$-Artinian modules and focus on this new class of modules. In the paper \cite{jNoet}, the definition of $\jmath$-submodule is given as follows: let $R$ be a commutative ring with identity, $M$ be an $R$-module and  $\jmath$ be a submodule of $M$. A submodule $N$ of $M$ is called a $\jmath$-submodule if $N \not \subseteq j $. We say that $M$ is called $\jmath$-Artinian $R$-module if every descending chain of $\jmath$-submodules is stationary. In this work, we characterize $\jmath$-Artinian modules. Furthermore, we give Akizuki theorem and its transition to amalgamation.
        Researchers introduced the notion of $\jmath$-Artinian rings in \cite{jArt} and obtained significant results concerning this new class of rings. Motivated by their definition and findings, we extend the study to modules by introducing the concept of $\jmath$-Artinian modules. Recall from \cite{jNoet} that, if $R$ is a commutative ring with identity, $M$ is an $R$-module, and $\jmath$ is a submodule of $M$, then a submodule $N$ of $M$ is called a $\jmath$-submodule if $N \not\subseteq \jmath$. We say that $M$ is a $\jmath$-Artinian $R$-module if every descending chain of $\jmath$-submodules becomes stationary. In this paper, we provide a characterization of $\jmath$-Artinian modules. Moreover, we establish an analogue of Akizuki’s theorem in this context and discuss its extension to amalgamated structures.
	}
	
    
\section{Introduction}
Throughout this paper, we focus only on commutative rings with identity
 and on unitary $R$-modules. $R$ always denote such a ring, and $M$ such an $R$-module.
Emil Artin significantly contributed to the advancement of the structure theory of commutative rings by introducing the theory of rings  satisfying descending chain conditions.  Recall that an $R$-module $M$  is called  Artinian if it satisfies the descending chain condition on submodules, and a ring $R$ is called Artinian if $R$ itself is an Artinian $R$-module, that is any descending chain of ideals of $R$ is stationary. One of the roots of the theory of Artinian rings is the historical article \cite{ae27} of Artin in 1927. Thereafter Artinian rings were continuously studied and hence have become one of the central subjects in the study of ring and module theory. The theory of Noetherian rings has a history extending over more than one hundred years. Recall that  an $R$-module $M$  is called  Noetherian if it satisfies ascending chain condition on submodules, and a ring $R$ is called Noetherian if $R$ itself is a Noetherian $R$-module, that is, any ascending chain of ideals of $R$ is stationary. Noetherian and Artinian rings and modules, and their generalizations, are noteworthy and popular research topics, especially for researchers in the fields of abstract algebra and algebraic geometry. For some related  studies, see \cite{jNoeR,jArt,Badawi, jNoet,   nnmod}.

Recently, Alhazmy et al. \cite{jNoeR} introduced and studied the concept of $\jmath$-Noetherian  rings as a proper generalization of Noetherian rings. Let $R$ be a commutative ring with identity and $\jmath$ be an ideal of $R$.  An ideal $I$ of $R$ is called a $\jmath$-ideal if $I \nsubseteq \jmath$. Also, $R$ is said to be  $\jmath$-Noetherian if every ascending chain of $\jmath$-ideals is stationary. They  transferred several properties and characterizations  of Noetherian rings to $\jmath$-Noetherian rings. More precisely, they \cite{jNoeR} investigated $\jmath$-Noetherian rings via the Cohen-type theorem, the flat extension, decomposable ring, the trivial extension ring, the amalgamated duplication, the polynomial ring extension, and the power series ring extension. The special case when $\jmath$ is  the nil-radical of $R$ was studied before in \cite{Badawi, HB} (under the name of
nonnil-Noetherian rings).  Also Dabbabi and  Benhissi \cite{DB}   studied the special case when $\jmath$ is the Jacobson 
radical of $R$ (under the name of non-$\jmath$-Noetherian rings). The authors in \cite{jArt} also introduced and studied the concept of $\jmath$-Artinian  rings as a proper generalization of Artinian rings. Let $R$ be a ring  and $\jmath$ be an ideal of $R$. Then $R$ is said to be  $\jmath$-Artinian if every descending chain of $\jmath$-ideals is stationary. They have transferred several properties and characterizations  of Artinian rings to $\jmath$-Artinian rings. For instance, a $\jmath$-Artinian ring is $\jmath$-Noetherian,  every prime $\jmath$-ideal is maximal in a $\jmath$-Artinian ring, and a $\jmath$-Artinian ring has only finitely many maximal ideals. Also they proved that $R$ is Artinian if and only if $R$ is zero-dimensional and  $\jmath$-Artinian, if and only if $R$ is $\jmath$-Artinian and each prime ideal is a $\jmath$-ideal,  if and only if $R$ is $\jmath$-Artinian and $Nil(R) \nsubseteq \jmath$. 

The authors in \cite{jNoet}  introduce and study the notion of $\jmath$-Noetherian modules as a generalization of Noetherian modules and $\jmath$-Noetherian rings. Let $R$ be a ring, $M$ be an $R$-module and $\jmath$ be a submodule of $M$. A submodule $N$ of $M$ is said to be a $\jmath$-submodule if $N \not\subseteq j$. $M$ is called a $\jmath$-Noetherian $R$-module if every $\jmath$-submodule of $M$ is finitely generated. They  give Cohen-type and Eakin-Nagata
Formanek theorems and  the Hilbert basis theorem  for $\jmath$-Noetherian modules. Moreover, they  determine the conditions for modules to be $\jmath$-Noetherian modules over polynomial rings and power series rings.

Inspired by the concepts of $\jmath$-Artinian rings and $\jmath$-Noetherian modules discussed in previous paragraph, in this paper, we introduce and study the notion of $\jmath$-Artinian modules as a generalization of Artinian modules and $\jmath$-Artinian  rings.
Let $R$ be a ring, $M$ be an $R$-module and  $\jmath$ be a submodule of $M$. We say that $M$ is called a $\jmath$-Artinian $R$-module if every descending chain of $\jmath$-submodules is stationary. The $\jmath$-Artinian modules presented in this paper are a generalization of Artinian modules, and note that for $j=0$ the definition of $\jmath$-Artinian modules and Artinian modules coincide. It is obvious that every Artinian $R$-module is $\jmath$-Artinian but converse is not true generally. If we take a non-Noetherian local ring with unique maximal ideal $\jmath$ as a module over itself, then you can see that it is a $\jmath$-Artinian module but not Artinian, see \cite[Example 2.3]{jArt}. Moreover, let $M$ be an $R$-module and  $j_0, j_1$ are submodules of $M$ such that $j_0 \subseteq j_1$ then it is clear that if $M$ is a $j_0$-Artinian module then it is also a $j_1$-Artinian from the definition.\par

In this paper, in the section 2 we characterize $\jmath$-Artinian modules and prove the Akizuki theorem. Moreover, we examine properties of quotient modules of $\jmath$-Artinian modules. We also investigate the prime and maximal submodules of $\jmath$-Artinian modules and the connections between them. Furthermore, we determine under what conditions a $\jmath$-Artinian module will also be an Artinian module. Finally, we consider the transition to amalgamation of $\jmath$-Artinian modules.
\section{Properties and characterizations of $\jmath$-Artinian modules}
For an  $R$-module $M$, it is well known  that every descending chain of submodules of $M$ terminates after a finite number of steps if and only if every nonempty family of submodules of $M$ has a minimal element, if and only if  an arbitrary intersection of the family of submodules becomes a finite intersection. We begin this section by proving this result for $\jmath$-submodules. 
\begin{thm}\label{eqcon}
  Let $R$ be a ring, $M$ be an $R$-module and $\jmath$ be a submodule of $M$. Then, the following statements are equivalent:
  \begin{enumerate}
      \item Every descending chain of $\jmath$-submodules of $M$ is stationary;
      \item Every nonempty set of $\jmath$-submodules of $M$ has a minimal element;
      \item If $\{N_i : i \in I\}$ is a family of $\jmath$-submodules of $M$, then there exists a finite subset $K$ of $I$ such that $ \bigcap_{i \in I} N_{i}  =   \bigcap_{i \in K} N_{i}  $.
  \end{enumerate}
\end{thm}
\begin{proof}
  (1) $\implies (2):$  Let $N$ be a nonempty subset of $\jmath$-submodules of $M$ and suppose that which has no minimal elements. Then, there exists $N_1 \in N$ and that is not a minimal element. Hence, there is an $N_2 \in N$ such that $N_1 \supseteq N_2 $. Similarly, we can find $N_3$ such that $N_2 \supseteq N_3$ since $N_2$ is not a minimal element. Therefore, we get a descending chain $N_1 \supseteq N_2 \supseteq N_3 \supseteq ... \supseteq N_i \supseteq ...$ of $\jmath$-submodules of $M$ that is not stationary and this contradicts our assumption (1). Then, every nonempty set of $\jmath$-submodules of $M$ has a minimal element.\par
  
  (2) $\implies$ (3): In here, we might presume that $I$ is nonempty. Let we take $W= \bigcap_{i \in I} N_i$ and $Q=\bigcap_{i=1}^{t} N_i$. Since $Q$ is nonempty, $Q$ has a minimal element from by (2), let we take $P=\bigcap_{i \in K} N_i$ for a finite subset $K$ of $I$. It is easy to see that $W \subseteq P$. Moreover, $P \bigcap N_i=P \subseteq N_i$ for every $i \in I$ by the minimality of $P$. Therefore, $P \subseteq \bigcap_{i \in I} N_i=W$. As a result, $W=P$.\par

  (3) $\implies (1):$ Let $N_1 \supseteq N_2 \supseteq N_3 \supseteq ... \supseteq N_i \supseteq ...$ be a descending chain of $\jmath$-submodules of $M$. In here, $\bigcap_{i \in I} N_i=\bigcap_{i=1}^{t} N_i=N_t$ from by (3). Hence, $N_l=N_t$ for all $l \geq t$.
\end{proof}
 \begin{defn}
         An $R$-module $M$ with a submodule $\jmath$ is said to be $\jmath$-Artinian if one of equivalent conditions of Theorem \ref{eqcon} is satisfied.
    \end{defn}
\begin{prop}\label{jArtQu1}
Let $M$ be an $R$-module with a submodule $\jmath$. If $M$ is a $\jmath$-Artinian, then $M/j$ is an Artinian $R$-module.
\end{prop}
\begin{proof}
Let $N_1/j \supseteq N_2/j \supseteq ...$ be a descending chain of nonzero submodules of $M/j$. In here, $j \subset N_i$ but $N_i \not \subseteq j$ for every $i$ since each $Ni/j$ is nonzero so $N_1 \supseteq N_2 \supseteq ...$ is a descending chain of $\jmath$-submodules of $M$. Then, there exists $l \in \mathbb{N}$ such that $N_p=N_{p+1}$ for every $p \geq l$ since $M$ is $\jmath$-Artinian. This implies that, $N_p/j=N_{p+1}/j$ for every $p \geq l$. Therefore,  $M/j$ is Artinian. 
\end{proof}
\begin{prop}\label{jArtQu2}
  Let $M$ be a $\jmath$-Artinian $R$-module. Then $M/N$ is an Artinian $R$-module for every $\jmath$-submodule $N$ of $M$.
\end{prop}
\begin{proof}
 Let $N$ be a $\jmath$-submodule of $M$ and $K_1/N \supseteq K_2/N \supseteq ...$  be a descending chain of nonzero submodules of $M/N$. In that case, $K_1 \supseteq K_2 \supseteq ...$ is a descending chain of $\jmath$-submodules of $M$. Then, there exist $t \in \mathbb{N}$ such that $K_{t+1}=K_t$ since $M$ is $\jmath$-Artinian. As a result, $K_{t+1}/N=K_t/N$ so $M/N$ is Artinian.
\end{proof}
\begin{thm}[Akuzi Theorem]\label{Ak}
 Every cyclic $\jmath$-Artinian $R$-module is $\jmath$-Noetherian.   
\end{thm}
\begin{proof}
   Let $M$ be a cyclic $\jmath$-Artinian $R$-module. Then, $M/N$ is cyclic Artinian $R$-module for every $\jmath$-submodule $N$ of $M$ by Proposition \ref{jArtQu2}. In that case, $M/N$ is Noetherian $R$-module for every $\jmath$-submodule $N$ of $M$. Therefore, $M$ is a $\jmath$-Noetherian $R$-module by \cite[Theorem 2.1]{jNoet}.
\end{proof}
\begin{prop}\label{primax}
Let $N$ be a proper $\jmath$-submodule of a $\jmath$-Artinian $R$-module $M$. $N$ is a prime submodule of $M$ if and only if $(N:M)$ is a maximal ideal of $R$.
\end{prop}
\begin{proof}
($\Rightarrow$ :) Assume that $N$ be a $\jmath$-prime submodule of $\jmath$-Artinian $R$-module $M$. Then, $M/N$ is an Artinian prime $R$-module by Proposition \ref{jArtQu2} and \cite[Lemma 2.3]{prime}. Let $S$ be set of all nontrivial submodules of $M/N$. Then, $S$ has a minimal element since $M/N$ is Artinian. Suppose that $L$ is a minimal element of $S$. It is clear that $L$ is a nonzero simple submodule of $M/N$. 
 Thus, there is a nonzero element $a+N \in M/N$ such that $L=R(a+N) \cong R/Ann(a+N)$ where $Ann(a+N)$ is a maximal ideal of $R$ by \cite[Lemma 1.1-(4)]{prime}. In here, $Ann(a+N)=Ann(M/N)=(N:M)$ since $M/N$ is a prime module (see \cite{prime}). Hence, $(N:M)$ is a maximal ideal of $R$. \par
 ($\Leftarrow$ :) Suppose that $(N:M)$ be a maximal ideal of $R$. See \cite[Lemma 2.1-(i)]{Azizi}.
\end{proof}
\begin{prop}\label{jprimax}
 Every $\jmath$-prime submodule of a $\jmath$-Artinian $R$-module $M$ is maximal.   
\end{prop}
\begin{proof}
 Let $M$ be a $\jmath$-Artinian $R$-module and $N$ be a $\jmath$-prime submodule of $M$. Suppose that $0 \neq K/N$ is a submodule of $M/N$. Then, $N \subsetneq K \subseteq M$. It is easy to see that $(N:M) \subsetneq (K:M)$. Then, $(K:M)=R$ since $(N:M)$ is a maximal ideal of $R$ by Proposition \ref{primax}. Hence, $1_R \in (K:M)$ so $1_RM=M \subseteq K$. This implies that, $M=K$. Thus, $K/N=M/N$. As a result, $M/N$ is a simple $R$-module so $N$ is a maximal submodule of $M$.
\end{proof}
\begin{prop}\label{finitemax}
 A $\jmath$-Artinian $R$-module has only a finite number of maximal submodules.
\end{prop}
\begin{proof}
 \textbf{Case 1:} Let $S$ be the set of all finite intersections $\cap_{i=1}^{t}N_{i}$ where $N_i$ are maximal $\jmath$-submodules of $M, \, 1 \leq i \leq t$. Then, $S$ has a minimal element by Theorem \ref{eqcon}, take $\cap_{j=1}^{l}N_{j}$ for $1 \leq j \leq l$. If we consider for any maximal $\jmath$-submodule $N$, then we can see that $N \cap N_1 \cap ... \cap N_l=N_1 \cap ... \cap N_l$ so $N \supseteq N_1 \cap ... \cap N_l$. Thus, $N \supseteq N_j$ for some $j$, then $N=N_j$ since $N_j$ is maximal. Therefore, $M$ has only finitely many maximal $\jmath$-submodules.\par
 \textbf{Case 2:} Let $N$ be an arbitrary maximal submodule of $M$. In here, if $\jmath$ is not a maximal submodule, then $N$ is a $\jmath$-submodule so  the desired result is obtained from case 1. Moreover, if $\jmath$ is a maximal submodule, then the only possible maximal submodule that is not a $\jmath$-submodule is $\jmath$. Therefore, $M$ has finite number of maximal submodules.
\end{proof}
\begin{prop}\label{local}
Let $M$ be an $R$-module and $S$ be a multiplicatively closed subset of $R$. If $M$ is $\jmath$-Artinian, then $S^{-1}M$ is $S^{-1}j$-Artinian $S^{-1}R$-module.    
\end{prop}
\begin{proof}
 Let $f:M \longrightarrow S^{-1}M, \, f(m)=\frac{m}{1}$. It is easy to see that $f$ is an $R$-module homomorphism so if $N \subseteq S^{-1}M$ is a submodule over $S^{-1}R$, then $L=f^{-1}(N)$ that the inverse image of a submodule of $S^{-1}M$ is a submodule of $M$. Moreover, it is clear that $S^{-1}L=N$. Thus, if $L$ is a $\jmath$-submodule of $M$, then $N=S^{-1}L \not\subseteq S^{-1}j$, i.e., it is a $S^{-1}j$-submodule of $S^{-1}M$. Similarly, it is trivial that if $N=S^{-1}L $ is a  $S^{-1}j$-submodule of $S^{-1}M$, then $L$ is a $\jmath$-submodule of $M$. Assume that, $N_1 \supseteq N_2 \supseteq ...$ is a descending chain of $S^{-1}j$-submodules of $S^{-1}M$ over $S^{-1}R$. Let $L_t=f^{-1}(N_t)$ for each $t \in \mathbb{N}^*$. Then, there exist $t \in \mathbb{N}^*$ such that $L_i=L_t$ for all $i \geq t$ since $M$ is $\jmath$-Artinian. Hence, $N_i=S^{-1}L_i=S^{-1}L_t=N_t$ for all $i \geq t$. As a result, $S^{-1}M$ is a $S^{-1}j$-Artinian $S^{-1}R$-module.
\end{proof}
\begin{thm}\label{exact}
 Let $0 \rightarrow M^{\prime} \xrightarrow{f} M \xrightarrow{g} M^{\prime \prime} \rightarrow 0$ be an  exact sequence of $R$-modules and $\jmath$ be a proper submodule of $M$. If $M$ is $\jmath$-Artinien  then  $M^{\prime } $ is $f^{-1}(j)$-Artinian and $M^{\prime \prime} $ is  $g(j)$-Artinian.  In addition, if $ker(g)\subseteq j$, then   the converse holds.
\end{thm}
\begin{proof} Suppose that $M$ is a $\jmath$-Artinian $R$-module. Let $N_1 \supseteq N_2 \supseteq \cdots \supseteq N_n \supseteq \cdots$ be a descending chain of  $g(j)$-submodules of $M^{\prime \prime}$. Then, for every $k\geq 1$ there exists $x_k\in M^{\prime }$  such that  $ y_k=g(x_k) \in N_k \backslash g(j)$. In that case, $x_k \in g^{-1}(N_k) \backslash  j$ so $g^{-1}(N_k)$ is a $\jmath$-submodule of $M$. Therefore, $ g^{-1}(N_1) \supseteq g^{-1}(N_2) \supseteq \cdots \supseteq g^{-1}(N_n) \supseteq \cdots$ is stationary    since $M$ is a $\jmath$-Artinian $R$-module. Hence,   $ N_1=g(g^{-1}(N_1))\supseteq N_2=g(g^{-1}(N_2)) \supseteq \cdots \supseteq N_n=g(g^{-1}(N_n)) \supseteq \cdots$ is stationary  so $M^{\prime \prime}$ is $g(j)$-Artinian $R$-module. Now  Let  $K_1 \supseteq K_2 \supseteq \cdots \supseteq K_n \supseteq \cdots$ be a descending chain of $f^{-1}(j)$-submodules of $M^{\prime}$. Then, there  for every $i\geq 1$ there exist $x_i\in K_i$  such that  $ x_i\notin f^{-1}(j)$. In that case, $f(x_i) \in f(K_i) \backslash  j$ so $f(K_i)$ is a $\jmath$-submodule of $M$ for every $k\geq 1$. Therefore, $f(K_1) \supseteq f(K_2) \supseteq \cdots \supseteq f(K_n) \supseteq \cdots$ is stationary  as $f$ is a monomorphism we conclude that $K_1 \supseteq K_2 \supseteq \cdots \supseteq K_n \supseteq \cdots$ is stationary  so $M^{\prime}$ is $f^{-1}(j)$-Artinian $R$-module. \par
Conversely, assume that $M^{\prime } $ is $f^{-1}(j)$-Artinian, $M^{\prime \prime} $ is  $g(j)$-Artinian and $ker(g)\subseteq j$. We can easily see that $g^{-1}(g(j))=j$ and so if $N$ is a $\jmath$-submodule of $M$, then $g(N)$ is  $g(j)$-submodule of $M^{\prime \prime}$. Now  let  $N_1 \supseteq N_2 \supseteq \cdots \supseteq N_n \supseteq \cdots$ be a descending chain of  $\jmath$-submodules of $M$, then $g(N_1) \supseteq g(N_2) \supseteq \cdots \supseteq g(N_n) \supseteq \cdots$ is a descending chain of  $g(j)$-submodules of $M^{\prime \prime}$ that is stationary, and also  $ f^{-1}(N_1) \supseteq f^{-1}(N_2) \supseteq \cdots \supseteq f^{-1}(N_n) \supseteq \cdots$ is also stationary by assumption. For large enough $n$ both these 
chains are stationary, and it follows that the chain $N_1 \supseteq N_2 \supseteq \cdots \supseteq N_n \supseteq \cdots$  is stationary.
\end{proof}
\begin{cor}
    Let $R$ ring,  $M$ is an $R$-module  and $N $ be a Noetherian  submodule of $\jmath$. Then, $M$ is a $\jmath$-Noetherian $R$-module if and only $M/N$ is  $j/N$-Noetherian $R$-module.
\end{cor}
\begin{proof}
  Accordibg to Theorem \ref{exact}, it is enough to show that, if $M/N$ is  $j/N$-Noetherian $R$-module, then $M$ is a $\jmath$-Noetherian $R$-module.   Let $K$ be a $\jmath$-submodule of $M$, then there exists an $x\in K \backslash j $, and so $x + N \in K + N \backslash (j+ N)$ since $N\subseteq j$. That is $K/N$ is  a $j/N$-submodule of $M/N$. Therefore  $K/N$ is a finitely generated submodule of $M/N$.    Set $K/N=\sum_{i=1}^t \operatorname{R}a_i + N$, where each $a_i \in K$. Let $x \in K$. Then, $x+N=\sum_{i=1}^t r_i a_i + N$ with $r_i \in R$.  $x-\sum_{i=1}^t r_i a_i \in K\cap N$. Now as $K\cap N \subseteq N$ and $N$ is Noetherian $R$-module we conclude that $K\cap N$ is finitely generated that is  $ K\cap N=\sum_{i=t+1}^{l} R a_i$ for some $a_{t+1}, a_{t+2}, \ldots, a_{l} \in K \cap N$.   Hence, $x=\sum_{i=1}^{t+l} r_i a_i$. Therefore, $K$ is finitely generated  and so $M$ is a $\jmath$-Noetherian $R$-module.
\end{proof}

\begin{prop}
  Let $n \geq 2$ be an integer, $M_1, \ldots, M_n$ $R$-modules, and let $\jmath_1, \ldots, \jmath_n$ be  a proper submodules  of  $M_1, \ldots, M_n$  respectively. Set $\jmath=\bigoplus_{i=1}^n \jmath_i $. Then the following assertions are equivalent:
		\begin{enumerate}
		\item $\bigoplus_{i=1}^n M_i$ is an  Artinian $R$-module;
			\item  $\bigoplus_{i=1}^n M_i$ is a $\jmath$-Artinian $R$-module;
			\item  For all $i=1, \ldots, n$, $M_i$ is an Artinian $R$-module.
		\end{enumerate}
\end{prop}

\begin{proof} By induction, it suffices to show the result for $n=2.$\par
    $(1) \Rightarrow (2)$: Straightforward.\par
     $(2) \Rightarrow (3)$:   Without loss of generality, we will show that $M_1$ is an  Artinian $R$-module. Let 
      Let $N_1 \supseteq N_2 \supseteq \cdots \supseteq N_n \supseteq \cdots$ be a descending chain of submodules of $M_1$.
 Since $ j_2\neq M_2 $, so $N_n \oplus M_2$ is a $\jmath$-submodule of $M$ for every $n$. Since $M$ is a $\jmath$-Artinian $R$-module, we get that descending chain of $\jmath$-submodules $N_1  \oplus M_2 \supseteq N_2 \oplus M_2 \supseteq \cdots \supseteq N_n  \oplus M_n \supseteq \cdots$ is stationary. Thus   $N_1 \supseteq N_2 \supseteq \cdots \supseteq N_n \supseteq \cdots$  is stationary, as required.\par

      $(3) \Rightarrow (1)$: This follows from \cite[Corollary 2.8.11]{WK} 
\end{proof}
 Note that the ideals of a decomposable ring  $R=R_1 \times R_2$ have the forms $I=I_1 \times I_2$ where 
$I_i$ is an ideal of $R_i$ for $i=1,2$. However, if $M_1, M_2$ are $R$-modules,  the  submodule of $M_1 \oplus M_2$ need not be of the form $P \oplus Q$. For example, inside $\mathbf{Z} \oplus \mathbf{Z}$ is the $\mathbf{Z}$-submodule $\mathbf{Z}(1,1)=\{(n, n): n \in \mathbf{Z}\}$.
%%%%%%%%%%%%
\begin{rem}
    	Let $M=M_1 \oplus M_2$ be an $R$-module and  $\jmath=\jmath_1 \oplus \jmath_2 $ be a submodule  of $M$. Suppose that  $ \jmath_1 =M_1$ or $ \jmath_2 =M_2$. Without loss of generality,  assume that  $ \jmath_1 =M_1$, and consequently $ \jmath_2 \neq M_2$. It is easy  to verify that an submodule $I=N_1\oplus N_2$ of $M$ is a $\jmath$-submodule of $M$  if and only if  $N_2$ is a $\jmath_2$-submodule of $M_2$. Let $N_1 \supseteq N_2 \supseteq \cdots \supseteq N_n \supseteq \cdots$ be a descending chain of $\jmath_2$-submodules of $M_2$, then $0\times N_1 \supseteq 0\times N_2 \supseteq \cdots \supseteq 0\times N_n \supseteq \cdots$ is  a descending chain of $\jmath$-submodules  of $M$ and so it is  stationary. Thus  $ N_1 \supseteq  N_2 \supseteq \cdots \supseteq  N_n \supseteq \cdots$ is stationary. Therefore   $M_2$ is a $j_2$-Artinian module.
 On the other hand, for every submodule $N_1$ of $M_1$, we have $N=N_1\times M_2$ is a $\jmath$-submodule of $M$, and so likewise  we deduce that   $M_1$ is an Artinian $R$-module. Conversely, assume that $M_1$ is an artinian  module and  $M_2$ is a $\jmath_2$-Artinian module.\\
 Thus   using the following natural exact sequence 
        $$0 \rightarrow M_2  \xrightarrow{f} M \xrightarrow{g} M_1 \rightarrow 0.$$
         Then $ker(g)= 0 \oplus M_2\subseteq j$ and so  Thus $M$ is $\jmath$-Artinian module according to Theorem \ref{exact}.

\end{rem}
%%%%%%%%%%%%%%%%%%%%%%%

Let $R$ be an integral domain and $M$ be a module over $R$. $M$ is said to be a divisible $R$-module, if $aM=M$ for every $a \neq 0 \in R$ \cite{Matlis}.
\begin{thm}\label{cong}
Let $M$ be a divisible finitely generated non torsion $R$-module and $\jmath$ be a prime submodule of $M$ such that $j \subseteq T(M)$. Then following statements are equivalent:
\begin{itemize}
    \item[(i)] $M$ is Artinian $R$-module;
    \item[(ii)] $M$ is $\jmath$-Artinian $R$-module.
\end{itemize}
\end{thm}
\begin{proof}
(i) $\implies$ (ii): It is trivial. \par
(ii) $\implies$ (i): Assume that $M$ is $\jmath$-Artinian $R$-module. Let $a$ be a nonzero element of $R$. Since $M$ is non torsion and $j \subseteq T(M)$, there exists $m \in M-j$ such that $ann(m)=0$. Then, $Ram \supseteq Ra^{2}m \supseteq ...$ is a descending chain of $\jmath$-submodules of $M$ since $M$ is divisible and $\jmath$ is prime submodule of $M$. In that case, there exists $k \in \mathbb{N}$ such that $Ra^km=Ra^{k+1}m$ since $M$ is $\jmath$-Artinian. Thus, there exists $r \in R$ such that $1_Ra^km=a^km=ra^{k+1}m$. This implies that, $m[a^k-ra^{k+1}]=0$. Therefore, we have $a^k[1_R-ra]=0$ since $ann(m)=0$.  Then, it must be $ra=ar=1_R$ since M is divisible so $R$ is integral domain and $a^k \neq 0$. Hence, $a \in U(R)$. As a result, $R$ is a field so it is an Artinian ring. Thus, $M$ is an Artinian $R$-module since $M$ is finitely generated module over Artinian ring $R$.
\end{proof}
\begin{cor}\label{corcong}
Let $M$ be a $\jmath$-Artinian divisible finitely generated non torsion $R$-module for prime submodule $j \subseteq T(M)$. Then, $M$ is Noetherian.     
\end{cor}
\begin{proof}
 Let $M$ be a $\jmath$-Artinian divisible finitely generated non torsion $R$-module for prime submodule $j \subseteq T(M)$. Then, $M$ is Artinian $R$-module from Theorem \ref{cong} so it is Noetherian. Moreover, $M$ is also $\jmath$-Noetherian since $M$ is Noetherian.   
\end{proof}
\begin{lem}\label{primej}
 Let $M$ be a $\jmath$-Artinian divisible $R$-module with non prime proper submodule $\jmath$. Then, every prime submodule of $M$ is a $\jmath$-submodule.   
\end{lem}
\begin{proof}
  Let $N$ be a prime submodule of $M$ and $am \in N$ for some $a \in R$ and $m \in M$. Suppose that $N \subseteq j$. This implies that, $m \in N$ since $N$ is prime, $M$ is divisible so $a \not \in (N:M)$. In that case, we have $m \in j$ when $am \in j$. Thus, $\jmath$ is a prime submodule of $M$. It contradicts with our assumption. Hence, $N \not \subseteq j$, i.e., $N$ is a $\jmath$-submodule of $M$.  
\end{proof}
\begin{thm}[\cite{Kasch}]\label{semisimple}
    An $R$-module $M$ is called semisimple if it satisfies one of the following equivalent definitions:
\begin{itemize}
    \item[(i)] $M$ is a direct sum of simple modules;
    \item[(ii)] $M$ is a sum of simple submodules;
    \item[(iii)] Every submodule of $M$ is a direct summand.
\end{itemize}
\end{thm}
\begin{thm}\label{Nil(M)j}
Let $M$ be a $\jmath$-Artinian divisible $R$-module with non prime proper submodule $\jmath$. If $M$ is faithful multiplication and $M/j$ is semisimple $R$-module, then $Nil(M) \subseteq j$.    
\end{thm}
\begin{proof}
 Since $M$ is a $\jmath$-Artinian $R$-module, $M$ has finitely many maximal submodules, say, $N_1,N_2,...,N_t$ by Proposition \ref{finitemax}. Assume that $j \subseteq N_1,N_2,...,N_s$ for $s<t$ and $j \not \subseteq N_{s+1},N_{s+2},...,N_t$. Then, $rad(M/j)=(N_1/j) \cap (N_2/j) \cap ... \cap (N_s/j)=(N_1 \cap N_2 \cap ... \cap N_s)/j$. Moreover, $rad(M/j)=\{0+j\}$ since $M/j$ is semisimple so it is also Artinian. This implies that, $N_1 \cap N_2 \cap ... \cap N_s=j$. Thus, $N_1 \cap N_2 \cap ... \cap N_s \cap N_{s+1} \cap ... \cap N_t \subseteq j$. As a result, 
 $$Nil(M)= \bigcap_{P \, \text{Prime}}P= \bigcap_{P \, j\text{-Prime}}P = \bigcap_{i=1}^{t}N_{i}  \subseteq j$$ by Proposition \ref{jprimax}, Proposition \ref{finitemax}, Lemma \ref{primej} and \cite[Theorem 6-(2)]{Mali}.
\end{proof}
Let $R$ be a ring and $M$ be a left $R$-module. A proper submodule $N$ of $M$ will be called virtually maximal if $M/N$ is a direct sum of isomorphic simple modules \cite{Rocky}.
\begin{thm}\label{virmax}
Let $M$ be a finitely generated module over a Noetherian ring $R$. If $M/N$ is semisimple for every proper submodule $N$ of $M$, then $M$ is Artinian $R$-module.      
\end{thm}
\begin{proof}
 Suppose that $M$ is an $R$-module with $M/N$ is semisimple for every proper submodule $N$ of $M$. $M$ is a Noetherian $R$-module since $M$ is finitely generated and $R$ is Noetherian. Moreover, every prime submodule $P$ of $M$ is virtually maximal since $M/P$ is semisimple by our assumption. Hence, $M$ is Artinian since $M$ is Noetherian and every prime submodule of $M$ is virtually maximal by \cite[Theorem 3.5]{Rocky}.
\end{proof}
\begin{cor}\label{jArtvir}
  Let $M$ be a finitely generated module over a Noetherian ring $R$. $M$ is a $\jmath$-Artinian $R$-module with $M/N$ is semisimple for every proper submodule $N$ of $M$ if and only if $M$ is Artinian $R$-module.        
\end{cor}
\begin{proof}
($\Leftarrow$) It is clear. \\
($\Rightarrow$) It is obvious by Theorem \ref{virmax}.
\end{proof}
\begin{cor}\label{jArtvir2}
$M$ is a $\jmath$-Artinian divisible finitely generated non torsion $R$-module for prime submodule $j \subseteq T(M)$ with $M/N$ is semisimple for every proper submodule $N$ of $M$ if and only if $M$ is Artinian $R$-module.        
\end{cor}
\begin{proof}
  ($\Leftarrow$) It is trivial. \par
($\Rightarrow$) It is obvious by Corollary \ref{corcong} and Theorem \ref{virmax}.  
\end{proof}
%%%%%%%%%%%%%%%%

\begin{prop}\label{jiArtinian}
  Let $M$ be an $R$-module, $\{\jmath_i\}_{i=1}^{s}$ be a finite family of submodules of $M$ and $\jmath=\cap_{i=1}^{s}j_i$. Then, the followings are equivalent:
 \begin{itemize}
     \item[(i)] $M$ is a $\jmath$-Artinian $R$-module,
     \item[(ii)] $M$ is a $\jmath_i$-Artinian $R$-module for all $i=1,2,...,s$.
 \end{itemize}  
\end{prop}
\begin{proof}
 (i) $\implies$ (ii) : Let $N$ be a $j_i$-submodule of $M$ for $i=1,2,...,s$. Then, there exists $n \in N-\jmath_i$. Thus, $N$ is a $\jmath$-submodule since $j=\cap_{i=1}^{s}j_i$ so $n \not \in \jmath$. Therefore, every $j_i$-submodule of $M$ is a $\jmath$-submodule. Since $M$ is a $\jmath$-Artinian $R$-module, every descending chain of $\jmath$-submodules of $M$ is stationary so every descending chain of $\jmath_i$-submodules of $M$ is also stationary. Hence, $M$ is a $j_i$-Artinian $R$-module for all $i=1,2,...,s$. \par
 (ii) $\implies$ (i) : Let $N_0 \supseteq N_1 \supseteq ...$ be a descending chain of $\jmath$-submodules of $M$. There exists $x_l \in N_l-j$ for every $l \in \mathbb{N}$ since $N_l$ is a $\jmath$-submodule of $M$. We claim that there exist $i_0 \in i=1,2,...,s$ such that all $N_l$ are $j_{i_0}$-submodules. Let's we accept the opposite situation. Then, there exists $h_i \in \mathbb{N}$ such that $N_{h_i}$ is not a $\jmath_i$-submodule for every $i=1,2,...,s$. Thus, $N_{h_i} \subseteq j_i$. Let $h$ be the largest of the $h_i$. In that case, we have that $N_h \subseteq j_i$ for every $i=1,2,...,s$ since $N_h \subseteq N_{h_i}$. This implies that, $N_h \subseteq j$ and it contradicts with $N_h$ is $\jmath$-submodule of $M$. Then, there exists $i_0 \in i=1,2,...s$ such that all $N_l$ are $j_{i_0}$-submodule so $N_0 \supseteq N_1 \supseteq ...$ is stationary as a descending chain of $j_{i_0}$-submodules by our assumption. As a result, $M$ is a $\jmath$-Artinian $R$-module.
\end{proof} 
For all the information given below regarding amalgamation and pullback, please refer to \cite{amal}.
Let $f:R \rightarrow S$ be a ring homomorphism,  $J$ be an ideal of $S$, $M$ be an $R$-module, $N$ be an $S$-module (which is an $R$-module induced naturally by $f$) and $\varphi: M \rightarrow N$ be an $R$-module homomorphism. Then, the amalgamation of $M$ and $N$ along $J$ with respect to $\varphi$ is $$M\bowtie^{\varphi}JN=\{(m,\varphi(m)+n):m\in M \,  \text{and} \, n \in JN\}.$$ Moreover, $M\bowtie^{\varphi}JN$ is an $R\bowtie^fJ$-module by the following scalar product $$(r,f(r)+j)(m,\varphi(m)+n)=(rm,\varphi(rm)+f(r)n+j\varphi(m)+jn).$$ Note that $\varphi(rm)=f(r)\varphi(m)$, since $\varphi$ is an $R$-module homomorphism. It can be seen that the following diagram is a pullback:
\[ \begin{tikzcd}
M\bowtie^{\varphi}JN \arrow{r}{\pi_N} \arrow[swap]{d}{\pi_M} & N \arrow{d}{\pi} \\%
M \arrow{r}{\pi \circ \varphi}& N/JN
\end{tikzcd}
\]
\begin{rem}\label{remark1}\cite[Remark 2.1- (3)]{amal} : 
 $M$ is an $R\bowtie^fJ$-module.   
\end{rem}
\begin{thm}\label{amaljArt}
Let $f:R \rightarrow S$  be a ring homomorphism, $\varphi:M \rightarrow N$ be an $R$-module homomorphism and $J$ be an ideal of $S$. Then, $M$ is $\jmath$-Artinian $R \bowtie^f J$-module if and only if $M \bowtie ^\varphi JN$ is a $j \bowtie^\varphi JN$-Artinian $R \bowtie^f J$-module.
\end{thm}
\begin{proof}
$(\Rightarrow:)$ Assume that $M$ is a $\jmath$-Artinian $R \bowtie^f J$-module. Let $K:=K_1  \supseteq K_2  \supseteq ...$ be a descending chain of $j \bowtie^\varphi JN$-submodules of  $R \bowtie^f J$-module $M \bowtie^\varphi JN$. Then, $\pi_M(K)$ is a descending chain of submodules of $M$ by \cite{amal}. Let $K_i^{'}=\pi_M(K)=\{m_i \in M : (m_i, \varphi(m_i)+n) \in K_i\}$. Since $M$ is $\jmath$-Artinian, there exists $m_{i_0} \in M-j$. In that case, $(m_{i_0},\varphi(m_{i_0})+n) \in K_{i_0}$ since $K_{i_0} \not \subseteq j \bowtie^\varphi JN$ so $K_{i_0}^{'}$ is a $\jmath$-submodule of $M$. Let $K^{'}:=K^{'}_{i_0} \supseteq K^{'}_{i_1} \supseteq ...$ be a descending chain of $\jmath$-submodules of $M$. Then, $K^{'}$ is stationary since $M$ is $\jmath$-Artinian. Hence, $K$ is also stationary, i.e., $M \bowtie^\varphi JN$ is a $j \bowtie^\varphi JN$-Artinian. \par
$(\Leftarrow):$ Suppose that $M \bowtie ^\varphi JN$ is a $j \bowtie^\varphi JN$-Artinian $R \bowtie^f J$-module. Let $N:=N_1 \supseteq N_2 \supseteq ...$ be a descending chain of $\jmath$-submodules of $M$. Then, there exists $m_i \in N_i-j$ for all $i$ since $N_i$ is $\jmath$-submodule of $M$. Set $N_i^{'}=\{(m_i,\varphi(m_i)+n) \in M \bowtie^\varphi JN : m_i \in N_i-j\}$.Then, it is easy to see that $N_i^{'}$ is a  $j \bowtie^\varphi JN$-submodule of $M \bowtie^\varphi JN$. Thus, $N^{'}:=N_1^{'} \supseteq N_2^{'} \supseteq ...$ is a descending chain of $j \bowtie^\varphi JN$-submodules of $M \bowtie^\varphi JN$. In that case, $N^{'}$ is stationary since $M \bowtie ^\varphi JN$ is a $j \bowtie^\varphi JN$-Artinian $R \bowtie^f J$-module. Then, $N$ is also stationary so $M$ is $\jmath$-Artinian.
\end{proof}

%%%%

%
\section*{Conflicts of Interest}
The authors declare that there are no conflicts of interest.

\end{document}